\newif{\ifarxiv}
\arxivtrue

\ifarxiv
\documentclass[12pt]{amsart}
\pdfoutput=1 
\else
\documentclass{amsart}
\fi
\usepackage[T1]{fontenc}   
\usepackage{amsfonts,amssymb}
\usepackage[all]{xy}

\ifarxiv
\relax
\else
\fi

\newtheorem{theorem}{Theorem}[section]
\newtheorem{lemma}[theorem]{Lemma}
\newtheorem{proposition}[theorem]{Proposition}
\newtheorem{fact}[theorem]{Fact}
\theoremstyle{definition}
\newtheorem{definition}[theorem]{Definition}

\newtheorem{remark}[theorem]{Remark}
\newtheorem{example}[theorem]{Example}
\numberwithin{equation}{section}

\newcommand\cb[1]{\mathbf{#1}} 
\newcommand{\Topcat}{\cb{Top}}
\newcommand{\Frmcat}{\cb{Frm}}
\newcommand{\Loccat}{\cb{Loc}}

\newcommand\real{\mathbb{R}}

\newcommand\rat{\mathbb{Q}}

\newcommand\nat{\mathbb{N}}

\newcommand\dc{\mathop{\downarrow}}
\newcommand\upc{\mathop{\uparrow}}
\newcommand\limp{\mathrel{\Rightarrow}}
\newcommand\uuarrow{\rlap{$\uparrow$}\raise.5ex\hbox{$\uparrow$}}
\newcommand\ddarrow{\rlap{$\downarrow$}\raise.5ex\hbox{$\downarrow$}}
\newcommand\directed{\sideset{}{^{\,\makebox[0pt]{$\scriptstyle\uparrow$}\!}}}
\newcommand\filtered{\sideset{}{^{\,\makebox[0pt]{$\scriptstyle\downarrow$}\!}}}
\newcommand\dsup{\directed\sup}

\newcommand\fcap{\filtered\bigcap}
\newcommand\eqdef{\mathrel{\buildrel \text{def}\over=}}
\newcommand\diff{\smallsetminus}

\newcommand\Vt{{\text{\textsf V}}}
\newcommand\Plotkin{\mathcal P} 
\newcommand\Plotkinn{\Plotkin_\Vt}
\newcommand\quasi{{\mathrm{quasi}}}

\newcommand\QLV{\Plotkin^\quasi_\Vt}
\newcommand\A{{\mathbf A}}
\newcommand\Aval{\Plotkin^\A}

\newcommand\TEMleq{\sqsubseteq^{\text{TEM}}}
\newcommand\EMleq{\sqsubseteq^{\text{EM}}}

\newcommand\pt{\text{\textsf{Spec}\;}}
\newcommand\Open{\mathcal O}

\begin{document}

\ifarxiv
\relax
\else
\noindent                                             
\begin{picture}(150,36)                               
\put(5,20){\tiny{Submitted to}}                       
\put(5,7){\textbf{Topology Proceedings}}              
\put(0,0){\framebox(140,34){}}                        
\put(2,2){\framebox(136,30){}}                        
\end{picture}                                        
\vspace{0.5in}

\renewcommand{\bf}{\bfseries}
\renewcommand{\sc}{\scshape}
\vspace{0.5in}
\fi

\title[Weakly Hausdorff Spaces]%
{On Weakly Hausdorff Spaces and Locally Strongly Sober Spaces}

\author{Jean Goubault-Larrecq}
\address{Universit\'e Paris-Saclay, CNRS, ENS Paris-Saclay, Laboratoire M\'ethodes Formelles, 91190, Gif-sur-Yvette, France.}
\email{goubault@lsv.fr}




\begin{abstract}
  We show that the locally strongly sober spaces are exactly the
  coherent sober spaces that are weakly Hausdorff in the sense of
  Keimel and Lawson.  This allows us to describe their Stone duals
  explicitly.  As another application, we show that weak Hausdorffness
  is a sufficient condition for lenses and of quasi-lenses to form
  homeomorphic spaces, generalizing previously known results.
\end{abstract}

\maketitle

\section{\bf Introduction}
\label{sec:intro}

In their study of measure extension theorems for $T_0$ spaces, Keimel
and Lawson introduced so-called \emph{weakly Hausdorff} spaces
\cite[Lemma~6.6]{KL:measureext}.  The notion does not seem to have
been investigated much.  We will show that these spaces are related to
the more well-known \emph{locally strongly sober} spaces: in
Section~\ref{sec:locally-strongly-sob}, we show that the latter are
exactly the weakly Hausdorff, coherent sober spaces.  This allows us
to elucidate their Stone duals in
Section~\ref{sec:stone-duals-locally}, as those frames in which the
filter-theoretic join of any two Scott-open filters is Scott-open.
The notion of weakly Hausdorff spaces also helps us generalize a few
results on variants of the so-called Plotkin powerdomain, namely on
spaces of lenses and of quasi-lenses, with shorter proofs, as we will
see in Section~\ref{sec:lens-quasi-lenses}.  An intermediate
Section~\ref{sec:remarks-examples} will be an opportunity of making
remarks and giving examples and counterexamples.  We start off
immediately with Section~\ref{sec:preliminaries}, giving the required
preliminaries.


\section{\bf Preliminaries}
\label{sec:preliminaries}

In general, we refer to \cite{GHKLMS:contlatt,JGL-topology} for
material on domain theory and topology.

A subset $A$ of a preordered set $(P, \leq)$ is upwards-closed if and
only if every element larger than or equal to an element of $A$ is in
$A$.  The \emph{upward closure} $\upc A$ of $A$ is the collection of
points that are larger than or equal to some point of $A$.  We write
$\upc x$ instead of $\upc \{x\}$.  We define downwards-closed sets,
downward closures $\dc A$, and $\dc x$ similarly.

Every topological space $X$ has a \emph{specialization preordering}
$\leq$, defined by $x \leq y$ if and only if every open neighborhood
of $x$ contains $y$.  
A space is $T_0$ if and only if the
specialization preordering $\leq$ is
antisymmetric.  

Conversely, every preordered set $(P, \leq)$ can be given the
\emph{Alexandroff topology}, whose open subsets are the upwards-closed
subsets of $P$.

The following is due to Keimel and Lawson
\cite[Lemma~6.6]{KL:measureext}.
\begin{definition}
  \label{defn:wH}
  A topological space $X$ is \emph{weakly Hausdorff} if and only if
  for all $x, y \in X$, every open neighborhood $W$ of
  $\upc x \cap \upc y$ contains an intersection $U \cap V$ of an open
  neighborhood $U$ of $x$ and of an open neighborhood $V$ of $y$.
\end{definition}
This should not be confused with the many other notions with similar
names \cite{Hoffmann:wH}, in particular with McCord's \emph{weak
  Hausdorff} spaces \cite{McCord:class:space}.

The following enumerates a few kinds of weakly Hausdorff spaces.
Items~(1) and~(3) are obvious, and item~(2) is Lemma~8.1 in
\cite{KL:measureext}.
\begin{proposition}
  \label{prop:wH}
  The following are weakly Hausdorff spaces:
  \begin{enumerate}
  \item all Hausdorff spaces; in fact, the Hausdorff spaces are exactly
    the $T_1$ weakly Hausdorff spaces;
  \item all stably locally compact spaces (see below) \cite{KL:measureext};
  \item all preordered sets in their Alexandroff topology.
  \end{enumerate}
\end{proposition}

We write $cl (A)$ for the closure of $A$ in $X$.  For any $x \in X$,
$cl (\{x\})$ is the downward closure $\dc x$ of $x$ in the
specialization preordering $\leq$.

A closed subset $C$ of a space $X$ is \emph{irreducible} if and only
if $C \neq \emptyset$ and
whenever $C$ intersects two open sets, it also intersects their
intersection.  A \emph{sober} space is a space in which every
irreducible closed set is the closure $\dc x$ of a unique point $x$.

A subset $Q$ of $X$ is compact if and only if one can extract a finite
cover from any open cover of $Q$.  We assume no separation property.
$X$ is \emph{locally compact} if and only if every point has a base of
compact neighborhoods.

A subset $A$ of a topological space is \emph{saturated} if and only if
it is equal to the intersection of its open neighborhoods, or
equivalently if and only if it is upwards-closed in the specialization
preordering $\leq$.  A space $X$ is \emph{coherent} if and only if the
intersection $Q \cap Q'$ of any two compact saturated subsets $Q$ and
$Q'$ is compact (and saturated).  A space is \emph{stably locally
  compact} if and only if it is locally compact, coherent, and sober.

A subset $D$ of a partially ordered set $(P, \leq)$ is \emph{directed}
if and only if every finite subfamily of $D$ has an upper bound in
$D$.  (In particular, $D$ is non-empty.)  $P$ is \emph{directed
  complete}, or a \emph{dcpo}, if and only if every directed family
$D$ in $P$ has a supremum $\dsup D$.  A \emph{Scott-open} subset of
$P$ is an upwards-closed subset $U$ of $P$ such that for every
directed family $D$, $\dsup D \in U$ implies that some element of $D$
is in $U$.  The Scott-open sets form the \emph{Scott topology}.

Every sober space is a \emph{monotone convergence} space
\cite[Exercise~O-5.15]{GHKLMS:contlatt}, namely a $T_0$ space such
that every directed subset $D$ 
has a supremum, and in which every open set is
Scott-open.

A subset $F$ of a partially ordered set $(P, \leq)$ is \emph{filtered}
if and only if it is directed in the opposite ordering.  A
\emph{filter} on $P$ is an upwards-closed filtered subset of $P$, and
it is a \emph{proper filter} if different from $P$.

Given a topological space $X$, let us write $\Open X$ for its set of
open subsets, ordered by inclusion.  A \emph{limit} of a filter
$\mathcal F$ on $\Open X$ is any point $x$ such that every open
neighborhood of $x$ is in $\mathcal F$.  We write $\lim \mathcal F$
for the set of limits of $\mathcal F$.  This is always a closed set.
Additionally, given any closed subset $C$ of $X$ in $\mathcal F$,
every limit of $\mathcal F$ is in $C$.

A \emph{filter base} is any filtered collection of non-empty subsets
of $X$.  By a standard use of Zorn's Lemma, every filtered base can be
completed to a maximal proper filter, namely to an \emph{ultrafilter}.

\section{\bf Locally Strongly Sober=Weakly Hausdorff+Coherent+Sober}
\label{sec:locally-strongly-sob}

A space $X$ is \emph{locally strongly sober} if and only if the set
$\lim \mathcal U$ of limits of every ultrafilter $\mathcal U$ on $X$
is either empty or the closure $\dc x$ of a unique point $x$
\cite[Definition~VI-6.12]{GHKLMS:contlatt}.

\begin{lemma}
  \label{lemma:lss=>wh}
  Every locally strongly sober space is weakly Hausdorff.
\end{lemma}
\begin{proof}
  Let $X$ be a locally strongly sober space, $x, y \in X$, and $W$ be
  an open neighborhood of $\upc x \cap \upc y$.  Let also
  $C \eqdef X \diff W$.  For the sake of contradiction, we assume that
  there is no pair of an open neighborhood $U$ of $x$ and of an open
  neighborhood $V$ of $Y$ such that $U \cap V \subseteq W$.  For all
  those pairs, $U \cap V \cap C$ is non-empty, so those sets form a
  filter base.  Let us extend it to an ultrafilter $\mathcal U$.
  Every open neighborhood of $x$ is in $\mathcal U$, so
  $x \in \lim \mathcal U$.  Similarly, $y \in \lim \mathcal U$.  Since
  $\lim \mathcal U$ is non-empty, by assumption there is a point
  $z \in X$ such that $\lim \mathcal U = \dc z$.  Now, on the one
  hand, $x$ and $y$ are in $\lim \mathcal U = \dc z$, so
  $z \in \upc x \cap \upc y \subseteq W$.  On the other hand, $C$ is
  closed and in $\mathcal U$, so every limit of $\mathcal U$ is in
  $C$.  In particular $z$ is in $C$, and that is impossible since $z$
  is in $W$.
\end{proof}

We write $\fcap_{i \in I} C_i$ for the intersection of any filtered
family ${(C_i)}_{i \in I}$.  A subset $Q$ of $X$ is compact if and
only if for every filtered family ${(C_i)}_{i \in I}$ of closed
subsets of $X$, if $Q \cap \fcap_{i \in I} C_i = \emptyset$ then
$Q \cap C_i = \emptyset$ for some $i \in I$.
\begin{lemma}
  \label{lemma:limU}
  For every ultrafilter $\mathcal U$ on a 
  space $X$, $\lim \mathcal U = \fcap_{A \in \mathcal U} cl (A)$.
\end{lemma}
\begin{proof}
  For every $A \in \mathcal U$, $cl (A)$ is a closed set in
  $\mathcal U$, so every limit of $\mathcal U$ is in $cl (A)$.
  Conversely, we show that every point $x \in X \diff \lim \mathcal U$
  is in the complement of $\fcap_{A \in \mathcal U} cl (A)$.  Since
  $x$ is not a limit of $\mathcal U$, there is an open neighborhood
  $U$ of $x$ that is not in $\mathcal U$.  Its complement
  $A \eqdef X \diff U$ is in $\mathcal U$, since $\mathcal U$ is an
  ultrafilter, and we conclude since $x$ is not in $cl (A) = A$.
\end{proof}

Let us introduce the following weakening of the notion of coherence.
\begin{definition}
  A space $X$ is \emph{weakly coherent} if and only if
  $\upc x \cap \upc y$ is compact for all points $x, y \in X$.
\end{definition}
Every coherent space, every $T_1$ space is weakly coherent.

\begin{lemma}
  \label{lemma:wh=>lss}
  Every weakly Hausdorff, weakly coherent, monotone convergence space
  $X$ is locally strongly sober.
\end{lemma}
\begin{proof}
  Let $\mathcal U$ be an ultrafilter on $X$, with
  $\lim \mathcal U \neq \emptyset$.
  For any two points $x, y \in \lim \mathcal U$, $\upc x \cap \upc y$
  is compact since $X$ is weakly coherent.  Let us assume that
  $\upc x \cap \upc y$ does not intersect $\lim \mathcal U$.  Since
  $\lim \mathcal U = \fcap_{A \in \mathcal U} cl (A)$ by
  Lemma~\ref{lemma:limU},
  $\upc x \cap \upc y \cap cl (A) = \emptyset$ for some
  $A \in \mathcal U$.  Let $W \eqdef X \diff cl (A)$. By weak
  Hausdorffness, there are an open neighborhood $U$ of $x$ and an open
  neighborhood $V$ of $y$ such that $U \cap V \subseteq W$.  Since
  $x \in \lim \mathcal U$, $U$ is in $\mathcal U$, and similarly $V$
  is in $\mathcal U$; therefore $U \cap V$, and then also $W$, is in
  $\mathcal U$.  But $A \in \mathcal U$, so $W \cap A$ is in
  $\mathcal U$; that is impossible, since $W \cap A = \emptyset$.

  This shows that $\lim \mathcal U$ is directed.  Since $X$ is a
  monotone convergence space, $\lim \mathcal U$ has a supremum $z$,
  and $z \in \lim \mathcal U$ because $\lim \mathcal U$ is closed.
  Every closed set is downwards-closed, so $\lim \mathcal U = \dc z$.
  The uniqueness of $z$ follows from the fact that every monotone
  convergence space is $T_0$.
\end{proof}

Every locally strongly sober space is sober
\cite[Lemma~VI-6.13]{GHKLMS:contlatt} and coherent
\cite[Lemma~VI-6.14]{GHKLMS:contlatt}.
This, together with Lemma~\ref{lemma:lss=>wh} and
Lemma~\ref{lemma:wh=>lss}, leads to the following.
\begin{theorem}
  \label{thm:lss=wh}
  For every topological space $X$, it is equivalent that $X$ be:
  \begin{enumerate}
  \item locally strongly sober;
  \item a weakly Hausdorff, weakly coherent, monotone
    convergence space;
  \item or a weakly Hausdorff, coherent, sober space.
  \end{enumerate}
\end{theorem}

\section{\bf Remarks and Examples}
\label{sec:remarks-examples}

\begin{remark}
  \label{rem:6.8}
  Theorem~6.8 of \cite{KL:measureext} states that given any weakly
  Hausdorff coherent sober space $X$, every locally finite inner
  regular valuation on the lattice of open subsets of $X$ extends to a
  $\sigma$-additive measure on a $\sigma$-algebra containing all the
  open sets and all the compact saturated sets.  By
  Theorem~\ref{thm:lss=wh}, the assumption on $X$ can be read as
  ``given any locally strongly sober space $X$''.
\end{remark}

\begin{remark}
  \label{rem:sloccomp}
  A space is stably locally compact if and only if it is locally
  compact and locally strongly sober
  \cite[Corollary~VI-6.16]{GHKLMS:contlatt}.  With
  Theorem~\ref{thm:lss=wh}, this allows us to give an equivalent
  definition as a conjunction of apparently weaker properties: a space
  is stably locally compact if and only if it is a core-compact,
  weakly Hausdorff, weakly coherent, monotone convergence space.  (A
  space is \emph{core-compact} if and only if its lattice of open sets
  is continuous.  A dcpo is \emph{continuous} if and only if every
  element is the supremum of a directed family of elements way-below
  it; the way-below relation $\ll$ is defined by $u \ll v$ if and only
  if every directed family $D$ such that $v \leq \bigvee D$ contains
  an element larger than or equal to $u$
  \cite[Section~I-1]{GHKLMS:contlatt}.  Every locally compact space is
  core-compact \cite[Example~I-1.7(5)]{GHKLMS:contlatt}, but not
  conversely \cite[Section~7]{HL:spectral}.  But every core-compact
  sober space is locally compact
  \cite[Theorem~V-5.6]{GHKLMS:contlatt}.)  I said ``apparently
  weaker'': in the presence of weak Hausdorffness, the following shows
  that monotone convergence equals sobriety, and that weak coherence
  equals coherence.
\end{remark}

\begin{proposition}
  \label{prop:monconv}
  Every weakly Hausdorff monotone convergence space $X$ is sober.
\end{proposition}
\begin{proof}
  Let $C$ be an irreducible closed subset of $X$.  For any two points
  $x$, $y$ of $C$, if $\upc x \cap \upc y$ were included in the
  complement $W$ of $C$, then $W$ would contain an intersection
  $U \cap V$ of an open neighborhood $U$ of $x$ and of an open
  neighborhood $V$ of $y$.  Since $C$ is irreducible, $U \cap V$, and
  therefore also $W$, would intersect $C$, but that is impossible.
  Hence $\upc x \cap \upc y$ intersects $C$.  This shows that $C$ is
  directed.  Since $X$ is a monotone convergence space,
  $z \eqdef \dsup C$ exists and is in $C$, so $C = \dc z$.  Since $X$
  is $T_0$, $z$ is unique, and therefore $X$ is sober.
\end{proof}

\begin{proposition}
  \label{prop:coh}
  Every weakly Hausdorff, weakly coherent space $X$ is coherent.
\end{proposition}
\begin{proof}
  Let $Q_1$, $Q_2$ be two compact saturated subsets of $X$, and let
  ${(W_i)}_{i \in I}$ be an open cover of $Q_1 \cap Q_2$.  For each
  pair $(x, y) \in Q_1 \times Q_2$, $\upc x \cap \upc y$ is included in
  $Q_1 \cap Q_2$, hence in $\bigcup_{i \in I} W_i$.  Since $X$ is
  weakly coherent, there is a finite subset $J_{xy}$ of $I$ such that
  $\upc x \cap \upc y \subseteq \bigcup_{i \in J_{xy}} W_i$.  Since
  $X$ is weakly Hausdorff, $x$ has an open neighborhood $U_x$ and $y$
  has an open neighborhood $V_y$ such that
  $U_x \cap V_y \subseteq \bigcup_{i \in J_{xy}} W_i$.  The sets
  $U_x$, $x \in Q_1$, form an open cover of $Q_1$, so there is a
  finite subset $E_1$ of $Q_1$ such that
  $Q_1 \subseteq \bigcup_{x \in E_1} U_x$.  Similarly, there is a
  finite subset $E_2$ of $Q_2$ such that
  $Q_2 \subseteq \bigcup_{y \in E_2} V_y$.  Then the sets $W_i$, where
  $i \in \bigcup_{x \in E_1, y \in E_2} J_{xy}$, form a finite subcover
  of $Q_1 \cap Q_2$.
%
%
\end{proof}

\begin{remark}
  \label{rem:jjl}
  Jia, Jung and Li show that every well-filtered weakly coherent dcpo
  is coherent in its Scott topology \cite{JJL:dcpo:coh}.  A space is
  \emph{well-filtered} if and only if given any filtered family
  ${(Q_i)}_{i \in I}$ of compact saturated subsets, and any open
  subset $U$, $\fcap_{i \in I} Q_i \subseteq U$ if and only if some
  $Q_i$ is included in $U$.  All sober spaces are well-filtered.
  Proposition~\ref{prop:coh} states a similar result, generalizing
  from dcpos to all topological spaces, and replacing
  well-filteredness by weak Hausdorffness.  Nonetheless, it is no
  generalization of the Jia-Jung-Li theorem, see
  Example~\ref{ex:Isbell}.
\end{remark}

The main purpose of the following examples is to show that
the terms ``weakly Hausdorff'', ``(weakly) coherent'', and ``sober''
(or ``monotone convergence space'') in Theorem~\ref{thm:lss=wh}~(3)
are irredundant.  They can more generally be used to get a better
grasp of the notions.

\begin{example}[Weakly Hausdorff, sober $\not\limp$ weakly coherent]
  \label{ex:Nab}
  Let us consider the dcpo $\nat \cup \{a, b\}$, where all the natural
  numbers are pairwise incomparable, and $a$ and $b$ are themselves
  incomparable and less than all natural numbers.  The Scott topology
  coincides with the Alexandroff topology, hence this dcpo is weakly
  Hausdorff.  It is also sober, but it is not weakly coherent, as
  $\upc a \cap \upc b = \nat$ is not compact.
\end{example}

\begin{example}[Weakly Hausdorff, coherent $\not\limp$ monotone
  convergence, sober]
  \label{ex:poset}
  Let us consider any dcpo $P$ in its Alexandroff (not Scott)
  topology.  If $P$ has an infinite ascending chain
  $x_0 < x_1 < \cdots < x_n < \cdots$, then $P$ is not a monotone
  convergence space, since the open set $\upc \dsup_{n \in \nat} x_n$
  is not Scott-open.  However, $P$ is weakly Hausdorff, and is
  coherent if and only if $\upc x \cap \upc y$ can be written as
  $\upc E$ for some finite set $E$, for all points $x$ and $y$.
  (Indeed, the compact subsets of $P$ are exactly the upward closures
  of finite sets.)  Hence, for example, $\nat \cup \{\omega\}$ (with
  $\omega > n$ for every $n \in \nat$) is weakly Hausdorff, coherent,
  but not a monotone convergence space in the Alexandroff topology of
  its natural ordering.
\end{example}

\begin{example}[Coherent $\not\limp$ sober, weakly Hausdorff]
  \label{ex:Ncof}
  The space $\nat$ with the cofinite topology is Noetherian (every
  subspace is compact), hence coherent, but neither sober nor weakly
  Hausdorff.  It is not sober because $\nat$ itself is irreducible
  closed.  It is not weakly Hausdorff because for any two distinct
  points $x$, $y$, $\upc x \cap \upc y = \{x\} \cap \{y\}$ is empty,
  but all open neighborhoods of $x$ and of $y$
  intersect.
\end{example}

\begin{example}[Sober, coherent $\not\limp$ weakly Hausdorff]
  \label{ex:aQ}
  A \emph{KC-space} is a space in which every compact set is closed.
  All KC-spaces are coherent, and also $T_1$ since one-element sets
  are compact, hence closed.
  Let $\alpha (\rat)$ denote $\rat \cup \{\infty\}$, the one-point
  compactification of the rational numbers.  Its open subsets are the
  open subsets of $\rat$, plus the sets $\alpha (\rat) \diff K$, where
  $K$ is compact in $\rat$.  This is a $T_1$, non-Hausdorff space
  \cite[Counterexample~35]{SS:countexamples}, hence is not weakly
  Hausdorff.  We argue that it is sober as follows.  Let $C$ be
  irreducible closed.  $C$ cannot contain two distinct rational
  numbers $x$ and $y$, otherwise there would be an open neighborhood
  $U$ of $x$ in $\rat$ and an open neighborhood $V$ of $y$ in $\rat$
  whose intersection is empty.
  If $C$ contains a rational number $x$ and $\infty$, therefore, $C$
  must be equal to $\{x, \infty\}$; then $C$ intersects
  ${]x-1, x+1[} \cap \rat$ and $\alpha (\rat) \diff \{x\}$, but not
  their intersection.  Hence $C$ contains only one point.  A result of
  Alas and Wilson states that the one-point compactification of a
  countable KC-space $X$ is a KC-space if and only if $X$ is
  sequential \cite[Corollary~7]{AW:KC}.  It follows that
  $\alpha (\rat)$ is a KC-space, hence is coherent.
\end{example}

\begin{example}[Coherent, well-filtered $\not\limp$ sober, weakly Hausdorff]
  \label{ex:Isbell}
  Isbell constructed a complete lattice $L$ that is not sober in its
  Scott topology \cite{Isbell:nonsober}.  $L$ is well-filtered,
  because every complete lattice is \cite{LX:wf}.  It is also weakly
  coherent, because every complete lattice is, trivially.  Hence it is
  coherent, by the Jia-Jung-Li theorem cited in Remark~\ref{rem:jjl}.
  However, it is not weakly Hausdorff, otherwise it would be sober by
  Theorem~\ref{thm:lss=wh}.
\end{example}


\section{\bf The Stone duals of locally strongly sober spaces}
\label{sec:stone-duals-locally}

For every topological space $X$, the set $\Open X$ of open subsets of
$X$ ordered by inclusion is a \emph{frame}, namely a complete lattice
such that
$u \wedge \bigvee_{i \in I} v_i = \bigvee_{i \in I} (u \wedge v_i)$
for all elements $u$ and $v_i$, $i \in I$.  This extends to a functor
from the category $\Topcat$ of topological spaces and continuous maps
to the opposite $\Loccat \eqdef \Frmcat^{op}$ of the category of
frames $\Frmcat$ (whose morphisms are the maps that preserve all
suprema and finite infima): for a continuous map $f \colon X \to Y$,
$\Open f$ maps $V \in \Open Y$ to $f^{-1} (V)$.

Conversely, given any frame $L$, a filter $\Phi$ on $L$ is
\emph{completely prime} if and only if
$\bigvee_{i \in I} u_i \in \Phi$ implies that $u_i \in \Phi$ for some
$i \in I$.  There is a sober space $\pt L$, whose points are the
completely prime filters, and whose open subsets are
$\Open_u \eqdef \{x \in \pt L \mid u \in x\}$, $u \in L$.  This
extends to a functor $\pt \colon \Loccat \to \Topcat$; for every frame
morphism $\varphi \colon L \to M$,
$\pt \varphi \colon \pt M \to \pt L$ maps every $y \in \pt M$ to
$\varphi^{-1} (y)$.  Additionally, $\Open$ is left-adjoint to $\pt$,
and this adjunction and cuts down to an equivalence between the full
subcategories of sober spaces and so-called spatial frames on the
other side.  (See Section~V-4 of \cite{GHKLMS:contlatt}, for example.)

We will identify the class of spatial frames that are related to the
locally strongly sober spaces through this equivalence.  In order to
do so, for any two filters $\mathcal F$, $\mathcal G$ on a frame $L$,
let us define the \emph{join} $\mathcal F \vee \mathcal G$ as
$\{u \wedge v \mid u \in \mathcal F, v \in \mathcal G\}$.  This is a
filter on $L$, and in fact the supremum of $\mathcal F$ and
$\mathcal G$ in the poset of all filters on $L$, ordered by inclusion.
\begin{definition}
  \label{defn:stable}
  A frame $L$ is \emph{locally temperate} if and only if the join
  $\mathcal F \vee \mathcal G$ of any two Scott-open filters on $L$ is
  Scott-open.  It is \emph{temperate} if and only if it is locally
  temperate and the smallest filter $\{\top\}$ (where $\top$ is the
  largest element of $L$) is Scott-open.
\end{definition}
Alternatively, a temperate frame is a frame in which the join of any
finite family of Scott-open filters is Scott-open.  We note that
$\{\top\}$ is Scott-open in $\Open X$ if and only if $X$ is compact.


We will rely on the Hofmann-Mislove theorem
\cite[Theorem~II-1.20]{GHKLMS:contlatt}: on a sober space $X$, the map
$\mathcal F \mapsto \bigcap \mathcal F$ defines an order-isomorphism
between the space of Scott-open filters on $\Open X$ and the space of
compact saturated subsets of $X$ ordered by reverse inclusion.  The
inverse map maps every compact saturated subset $Q$ to its set of open
neighborhoods $\blacksquare Q$.  We also rely on the following.
\begin{fact}[Lemma~6.6 of \cite{KL:measureext}]
  \label{fact:wH}
  A space $X$ is weakly Hausdorff if and only if for every pair of
  compact saturated subsets $Q_1$ and $Q_2$, for every open
  neighborhood $W$ of $Q_1 \cap Q_2$, there is an open neighborhood
  $U$ of $Q_1$ and and open neighborhood $V$ of $Q_2$ such that $U
  \cap V \subseteq W$.
\end{fact}

\begin{lemma}
  \label{lemma:wHc:stable}
  The following are equivalent for every sober space $X$:
  \begin{enumerate}
  \item $\Open X$ is locally temperate;
  \item $X$ is weakly Hausdorff and coherent;
  \item $X$ is locally strongly sober.
  \end{enumerate}
\end{lemma}
\begin{proof}
  $(2) \limp (1)$.  Let $\mathcal F$
  and $\mathcal G$ be two Scott-open filters on $\Open X$.  By the
  Hofmann-Mislove theorem, $\mathcal F = \blacksquare Q_1$ and
  $\mathcal G = \blacksquare Q_2$, for some compact saturated sets
  $Q_1$ and $Q_2$.  We claim that
  $\mathcal F \vee \mathcal G = \blacksquare {(Q_1 \cap Q_2)}$.  Every
  element of $\mathcal F \vee \mathcal G$ is of the form $U \cap V$
  where $U \in \blacksquare Q_1$ and $V \in \blacksquare Q_2$, hence
  belongs to $\blacksquare {(Q_1 \cap Q_2)}$.  Conversely, let
  $W \in \blacksquare {(Q_1 \cap Q_2)}$.  By Fact~\ref{fact:wH}, there
  is an open neighborhood $U$ of $Q_1$ and and open neighborhood $V$
  of $Q_2$ such that $U \cap V \subseteq W$.  In other words,
  $U \cap V$, and therefore also the larger set $W$, in in
  $\mathcal F \vee \mathcal G$.
  Finally, since $X$ is coherent, $Q_1 \cap Q_2$ is compact, so
  $\mathcal F \vee \mathcal G = \blacksquare {(Q_1 \cap Q_2)}$ is
  Scott-open.

  $(1) \limp (2)$.  Since $\Open X$ is locally temperate, for all compact
  saturated subsets $Q_1$ and $Q_2$ of $X$,
  $\blacksquare Q_1 \vee \blacksquare Q_2$ is Scott-open, hence is
  equal to $\blacksquare Q$ for some compact saturated set $Q$.  Then
  $Q$ is the intersection of its open neighborhoods
  $W \in \blacksquare Q$, hence is the intersection of the sets
  $U \cap V$ with $U \in \blacksquare Q_1$ and
  $V \in \blacksquare Q_2$, which is equal to $Q_1 \cap Q_2$.  It
  follows that $Q_1 \cap Q_2$ is compact.  Therefore $X$ is coherent.
  Using the same construction, we obtain that every open neighborhood
  $W$ of $Q_1 \cap Q_2$ can be written as $U \cap V$ with
  $U \in \blacksquare Q_1$ and $V \in \blacksquare Q_2$.  By
  Fact~\ref{fact:wH}, $X$ is weakly Hausdorff.

  $(2) \Leftrightarrow (3)$.  This is by Theorem~\ref{thm:lss=wh}.
\end{proof}

\begin{theorem}
  \label{thm:stable}
  The $\Open \dashv \pt$ adjunction cuts down to an adjoint
  equivalence between the full subcategories of (locally) strongly
  sober spaces and of (locally) temperate spatial frames.
\end{theorem}

Not all locally temperate frames are spatial.  In order to see this,
we first prove the following lemma.
\begin{lemma}
  \label{lemma:coframe}
  Every complete Boolean algebra is weakly temperate.
\end{lemma}
\begin{proof}
  Let $L$ be a complete Boolean algebra, and $\mathcal F$ and
  $\mathcal G$ be two Scott-open filters in $L$.  Let
  ${(u_i)}_{i \in I}$ be a directed family whose supremum is in
  $\mathcal F \vee \mathcal G$.  We can write this supremum as
  $v \wedge w$ where $v \in \mathcal F$ and $w \in \mathcal G$.  The
  family ${(u_i \vee (v \wedge \neg w))}_{i \in I}$ is directed, and
  its supremum is $(v \wedge w) \vee (v \wedge \neg w) = v$.  Since
  $v \in \mathcal F$ and $\mathcal F$ is Scott-open,
  $u_i \vee (v \wedge \neg w)$ is in $\mathcal F$ for some $i \in I$.
  Similarly, $u_j \vee (w \wedge \neg v)$ is in $\mathcal G$ for some
  $j \in I$.  By directedness, we may assume that $i=j$.  Then
  $(u_i \vee (v \wedge \neg w)) \wedge (u_i \vee (w \wedge \neg v))$
  is in $\mathcal F \vee \mathcal G$.  But
  $(u_i \vee (v \wedge \neg w)) \wedge (u_i \vee (w \wedge \neg v)) =
  u_i \vee ((v \wedge \neg w) \wedge (w \wedge \neg v)) = u_i$, so
  $u_i \in \mathcal F \vee \mathcal G$.
\end{proof}

\begin{remark}
  \label{rem:temp:notspatial}
  Not all locally temperate frames are spatial.  For a counterexample,
  consider the complete Boolean algebra of regular open subsets of
  $\real$.  This is weakly temperate by Lemma~\ref{lemma:coframe}, but
  does not have any point, hence is not spatial
  \cite[Exercise~8.1.25]{JGL-topology}.
\end{remark}

\begin{remark}
  \label{rem:stable:cont}
  Let us write $\uuarrow u$ for $\{v \in L \mid u \ll v\}$.  A frame
  is \emph{stable} if and only if $u \ll v, w$ implies
  $u \ll v \wedge w$ \cite[Definition~2.12]{BH:scomp:frames}.  A sober
  space $X$ is stably locally compact if and only if $\Open X$ is
  continuous and stable.  A compact, continuous, stable frame is
  called a \emph{stably continuous} frame by Banaschewski and
  Br{\"u}mmer \cite{BB:scont:frame}, or a \emph{stably compact} frame
  by Johnstone \cite{Johnstone:stone}.  We prefer the former name,
  since the latter seems to imply that the frame would be stably
  compact in its Scott topology, which would be wrong.  Using
  Theorem~\ref{thm:stable}, one can show the following localic
  analogue of the characterization of stably (locally) compact spaces
  given in Remark~\ref{rem:sloccomp}.  We give an elementary proof,
  which, as for several other results in locale theory, does not
  require the axiom of choice.
\end{remark}

\begin{proposition}
  \label{prop:scont:frame}
  A continuous frame $L$ is stable if and only if it is locally
  temperate.  The stably continuous frames are exactly the continuous,
  temperate frames.
\end{proposition}
\begin{proof}
  Let $L$ be a continuous frame.  We only need to show that it is
  stable if and only if it is locally temperate.  Since $L$ is
  continuous, every Scott-open subset $U$ is the union of sets
  $\uuarrow u$, $u \in U$, and those sets are themselves Scott-open
  \cite[Proposition~II-1.10]{GHKLMS:contlatt}.  Given any two
  Scott-open filters $\mathcal F$ and $\mathcal G$ of $L$,
  $\mathcal F \vee \mathcal G = \bigcup_{u \in \mathcal F} \uuarrow u
  \vee \bigcup_{v \in \mathcal G} \uuarrow v = \{u' \wedge v' \mid
  \exists u \in \mathcal F, v \in \mathcal G, u \ll u', v \ll v'\}$.
  For every directed family ${(w_i)}_{i \in I}$ whose supremum is in
  $\mathcal F \vee \mathcal G$, namely such that
  $\bigvee_{i \in I} w_i = u' \wedge v'$ for some $u \ll u'$ in
  $\mathcal F$ and some $v \ll v'$ in $\mathcal G$, we have
  $u \wedge v \ll u', v'$.  If $L$ is stable, then
  $u \wedge v \ll u' \wedge v'$, so $u \wedge v \leq w_i$ for some
  $i \in I$, and therefore $\mathcal F \vee \mathcal G$ is Scott-open.
  
  Conversely, if $L$ is locally temperate, we show that $L$ is stable
  as follows.  We use the following trick: since $L$ is continuous,
  for any two elements $u, v \in L$ such that $u \ll v$, $v$ belongs
  to a Scott-open filter $\mathcal F$ included in $\uuarrow u$, hence
  in $\upc u$ \cite[Proposition~I-3.3]{GHKLMS:contlatt}.  Given
  another element $w$ such that $u \ll w$, similarly, we find a
  Scott-open filter $\mathcal G$ containing $w$ and included in
  $\upc u$.  Since $L$ is locally temperate,
  $\mathcal F \vee \mathcal G$ is Scott-open.  It is included in
  $\upc u$, and contains $v \wedge w$.  It follows easily that
  $u \ll v \wedge w$.
\end{proof}

\section{\bf An application: on lenses and quasi-lenses}
\label{sec:lens-quasi-lenses}

One of the standard powerdomains considered in domain theory is the
so-called \emph{Plotkin powerdomain}
\cite[Definition~IV-8.11]{GHKLMS:contlatt}.  In specific situations,
this coincides with various other, more concrete constructions.  Let
us cite three, parameterized by a base space $X$.
\begin{itemize}
\item The space $\Plotkinn X$ of all lenses, with the Vietoris
  topology.  A \emph{lens} is a non-empty set of the form $Q \cap C$
  where $Q$ is compact saturated and $C$ is closed in $X$.  The
  \emph{Vietoris topology} is the familiar one on hyperspaces, and has
  subbasic open subsets of the form $\Box U$ (the set of lenses
  included in $U$) and $\Diamond U$ (the set of lenses that intersect
  $U$), for each open subset $U$ of $X$.
\item The space $\Aval X$ of $\A$-valuations \cite{Heckmann:absval}.
  We omit the definition.
\item The space $\QLV X$ of all quasi-lenses, with a topology which we
  will again call the Vietoris topology.  A \emph{quasi-lens} is a
  pair $(Q, C)$ where $Q$ is compact saturated, $C$ is closed, and the
  following three conditions are met:
  \begin{enumerate}
  \item $Q$ intersects $C$;
  \item $Q \subseteq \upc (Q \cap C)$;
  \item for every open neighborhood $U$ of $Q$, $C \subseteq cl (U
    \cap C)$.
  \end{enumerate}
  The subbasic open subsets are $\Box^\quasi U$ (the set of
  quasi-lenses $(Q, C)$ such that $Q \subseteq U$) and
  $\Diamond^\quasi U$ (the set of quasi-lenses $(Q, C)$ such that $C$
  intersects $U$), for all open subsets $U$ of $X$.
\end{itemize}
Quasi-lenses originate from Heckmann's work
\cite[Theorem~9.6]{heckmann:2ndorder}.  If $X$ is sober, then $\QLV X$
is homeomorphic to $\Aval X$ \cite[Fact~5.2]{GL:duality}.
It is known that $\QLV X$ (equivalently, $\Aval X$) is homeomorphic to
$\Plotkinn X$ when $X$ is Hausdorff
\cite[Theorem~5.1]{Heckmann:absval}, or when $X$ is stably compact
\cite[Proposition~5.3]{GL:duality}.  We show that both of these
results stem from a more general statement on weakly Hausdorff
spaces. The proof is also simpler than the one given in the stably
compact case \cite{GL:duality}.

The key to the new proof is the following lemma.
\begin{lemma}
  \label{lemma:C:lens:wH}
  For every compact saturated subset $Q$ of a weakly Hausdorff space
  $X$ and for every closed subset $C$ of $X$, if
  $C \subseteq cl (U \cap C)$ for every open neighborhood $U$ of $Q$,
  then $C \subseteq cl (Q \cap C)$.
\end{lemma}
\begin{proof}
  Let us imagine that $C$ is not included in $cl (Q \cap C)$, and let
  $y$ be a point that in $C$ and not in $cl (Q \cap C)$.  The
  intersection $Q \cap \upc y \cap C$ is empty, otherwise it would
  contain a point $z \in Q \cap C \subseteq cl (Q \cap C)$ such that
  $y \leq z$, and that would entail that $y$ is in $cl (Q \cap C)$.
  Therefore $Q \cap \upc y \subseteq W$, where $W$ is the complement
  of $C$.  Since $X$ is weakly Hausdorff, and using
  Fact~\ref{fact:wH}, there is an open neighborhood $U$ of $Q$ and an
  open neighborhood $V$ of $\upc y$ such that $U \cap V \subseteq W$.
  Since $Q \subseteq U$, by assumption $C \subseteq cl (U \cap C)$.
  Now $y$ is both in $C$ and in $V$, so $V$ intersects $C$, and
  therefore also the larger set $cl (U \cap C)$.  Since $V$ is open,
  it must also intersect $U \cap C$, so $U \cap V \cap C$ is
  non-empty.  Then $W \cap C$ is non-empty, which is impossible since
  $W$ is the complement of $C$.
\end{proof}

\begin{proposition}
  \label{prop:qlens=lens}
  Let $X$ be a topological space.
  \begin{enumerate}
  \item For every lens $L$, $\iota (L) \eqdef (\upc L, cl (L))$ is a
    quasi-lens.
  \item For every quasi-lens $(Q, C)$,
    $\varrho (Q, C) \eqdef Q \cap C$ is a lens.
  \end{enumerate}
  Additionally, $\varrho \circ \iota$ is the identity map, and
  $\iota \circ \varrho$ is the identity map if the conclusion of
  Lemma~\ref{lemma:C:lens:wH} holds, in particular if $X$ is weakly
  Hausdorff.
\end{proposition}
\begin{proof}
  (1) Let $L \eqdef Q \cap C$ be a lens.  In particular, $L$ is
  compact, so $Q' \eqdef \upc L$ is compact saturated.
  $C' \eqdef cl (L)$ is clearly closed, and $Q' \cap C'$ contains $L$,
  hence is non-empty.  For every $x \in Q'$, there is a $y \in L$ such
  that $y \leq x$.  Then $y$ is both in $Q'$ and in $C'$, so $x$ is in
  $\upc (Q' \cap C')$.  For every open neighborhood $U$ of $Q'$, $U$
  contains $L$, so $U \cap C'$ also contains $L$, hence
  $cl (U \cap C')$ contains $cl (L) = C'$.  Therefore $(Q', C')$ is a
  quasi-lens.

  (2) For every quasi-lens $(Q, C)$, $L \eqdef Q \cap C$ is clearly a
  lens.

  For every lens $L$, $\varrho (\iota (L)) = \upc L \cap cl (L)$ is
  equal to $L$; see for example
  \cite[Definition~IV-8.15]{GHKLMS:contlatt}.

  For the final part of the proposition, let $(Q, C)$ be any
  quasi-lens, $L \eqdef \varrho (Q, C) = Q \cap C$, and $(Q', C')$ be
  $\iota (L)$, namely $Q' \eqdef \upc L$ and $C' \eqdef cl (L)$.
  Since $L \subseteq C$ and $C$ is closed, $C' \subseteq C$.  The
  reverse inclusion is the conclusion of Lemma~\ref{lemma:C:lens:wH},
  so $C'=C$.  Since $L \subseteq Q$ and $Q$ is saturated,
  $Q' \subseteq Q$; but we also have
  $Q \subseteq \upc (Q \cap C) = \upc L = Q'$, by definition of a
  quasi-lens, so $Q'=Q$.  Therefore $\iota \circ \rho$ is the identity
  map.
\end{proof}


\begin{theorem}
  \label{thm:qlens=lens:topo}
  For every topological space $X$, the map $\iota$ is a topological
  embedding of $\Plotkinn X$ into $\QLV X$.  For every open subset $U$
  of $X$, $\iota^{-1} (\Box^\quasi U) = \Box U$ and
  $\iota^{-1} (\Diamond^\quasi U) = \Diamond U$.

  If the conclusion of Lemma~\ref{lemma:C:lens:wH} holds, notably if
  $X$ is weakly Hausdorff, then $\iota$ is a homeomorphism of
  $\Plotkinn X$ onto $\QLV X$, with inverse $\varrho$.
\end{theorem}
\begin{proof}
  For every lens $L$, $\iota (L) \in \Box^\quasi U$ if and only if
  $\upc L \subseteq U$, if and only if $L \in \Box U$.  Therefore
  $\iota^{-1} (\Box^\quasi U) = \Box U$.  Also,
  $\iota (L) \in \Diamond^\quasi U$ if and only if $cl (L)$ intersects
  $U$, if and only if $L$ intersects $U$, if and only if
  $L \in \Diamond U$.  Hence
  $\iota^{-1} (\Diamond^\quasi U) = \Diamond U$.

  The map $\iota$ is injective by Remark~IV-8.16 of
  \cite{GHKLMS:contlatt}, and what we have just shown then implies
  that $\iota$ is a topological embedding.
  The final claim follows from
  Proposition~\ref{prop:qlens=lens}.
\end{proof}

The specialization ordering of $\Plotkinn X$ is the so-called
\emph{topological Egli-Milner ordering}: $L \TEMleq L'$ if and only if
$\upc L \supseteq \upc L'$ and $cl (L) \subseteq cl (L')$
\cite[Discussion before Fact~4.1]{GL:duality}.  Instead, the
\emph{Egli-Milner ordering} is defined by $L \EMleq L'$ if and only if
$\upc L \supseteq \upc L'$ and $\dc L \subseteq \dc L'$.  When $X$ is
stably compact, those two orderings coincide
\cite[Lemma~4.2]{GL:duality}, and in fact $\dc L = cl (L)$ for every
lens $L$.  We show that this holds, more generally, on all weakly
Hausdorff spaces.  The proof, once again, is elementary.

\begin{theorem}
  \label{thm:dcL:closed}
  For every lens $L$ in a weakly Hausdorff space $X$, $\dc L$ is
  closed, and so $cl (L) = \dc L$.  The orderings $\TEMleq$ and
  $\EMleq$ coincide.
\end{theorem}
\begin{proof}
  Let us write $L$ as $Q \cap C$ where $Q$ is compact saturated and
  $C$ is closed.  We show that $\dc L$ is closed by showing that every
  point $x$ outside $\dc L$ lies in some open set disjoint from
  $\dc L$.  Since $x \not\in \dc L$, $\upc x$ is disjoint from $L$, so
  $\upc x \cap Q \cap C$ is empty.  Let $W$ be the complement of $C$.
  Then $\upc x \cap Q$ is included in $W$.  Since $W$ is weakly
  Hausdorff, there is an open neighborhood $U$ of $x$ and an open
  neighborhood $V$ of $Q$ such that $U \cap V \subseteq W$.  Then
  $U \cap L = U \cap Q \cap C \subseteq U \cap V \cap C \subseteq W
  \cap C = \emptyset$, so $U$ is disjoint from $L$, as desired.

  Finally, $\dc L \subseteq cl (L)$ since every closed set is
  downwards-closed.  We have just shown that $\dc L$ is closed.  It
  contains $L$, so it contains $cl (L)$.
\end{proof}

Weak Hausdorffness is not necessary in those results, as the following
example shows; but something is needed, see
Example~\ref{ex:Ncof:lens}.

\begin{example}
  \label{ex:aQ:lens}
  In any $T_1$ space, the lenses are exactly the non-empty compact
  subsets.  Although $\alpha (\rat)$ is not weakly Hausdorff (see
  Exercise~\ref{ex:aQ}), $\dc L$ is closed for every lens $L$ of
  $\alpha (\rat)$, because $\dc L = L$, and every compact set is
  closed.  Given any quasi-lens $(Q, C)$ in $\alpha (\rat)$,
  $Q \subseteq \upc (Q \cap C)$ means that $Q \subseteq C$.  We claim
  that $Q=C$.  This trivially entails that the conclusion of
  Lemma~\ref{lemma:C:lens:wH} holds, although $\alpha (\rat)$ is not
  weakly Hausdorff.

  In order to show the claim, let us assume $x \in C \diff Q$.  If
  $\infty \not\in Q$, then $Q$ is compact in $\rat$, and we can cover
  it by a finite union $U$ of open balls of radius $\epsilon > 0$,
  where $\epsilon$ is strictly less than the distance of $x$ to the
  closed set $Q$.  Let $K$ be the union of the corresponding closed
  balls.  The condition $C \subseteq cl (U \cap C)$ then entails
  $C \subseteq K \cap C$; this is impossible since $x \in C$ but
  $x \not\in K$.  If $\infty$ is in $Q$, then $x \neq \infty$, and we
  use a similar argument.  We pick $n > |x|$.  Then $Q \cap {[-n, n]}$
  is closed and bounded, hence compact in $\rat$.  We cover it by a
  finite union $U$ of open balls of radius $\epsilon > 0$, where
  $\epsilon$ is strictly less than the distance of $x$ to
  $Q \cap {[-n, n]}$ (if that set is non-empty, otherwise $\epsilon$
  is unconstrained).  We let $K$ be the union of the corresponding
  closed balls.  We note that $Q \subseteq V$ where
  $V \eqdef U \cup (\alpha (\rat) \diff {[-n, n]})$.  The condition
  $C \subseteq cl (V \cap C)$ entails
  $C \subseteq K \cup {]-\infty,-n]} \cup {[n, +\infty[} \cup
  \{\infty\}$; but $x$ is in $C$ and not in the right-hand side.
\end{example}


\begin{example}
  \label{ex:Ncof:lens}
  Let us consider $\nat$ with the cofinite topology (see
  Example~\ref{ex:Ncof}).  As in every $T_1$ space, the lenses $L$ are
  the non-empty compact subsets, and here this means the non-empty
  subsets.  Then $cl (L) = \dc L$ if and only if $L$ is finite or
  $L=\nat$, so $cl (L) \neq \dc L$ for every infinite proper subset
  $L$ of $\nat$.  We have $L \TEMleq L'$ if and only if $L=L'$ is
  finite, or $L'$ is infinite and is included in $L$;
  while $L \EMleq L'$ if and only if $L=L'$.  Therefore the relations
  $\TEMleq$ and $\EMleq$ differ.  The quasi-lenses are exactly the
  pairs $(Q, C)$ where $Q=C$ is finite and non-empty, or where
  $Q \neq \emptyset$ and $C = \nat$.  However, the image of $\iota$
  consists of all pairs $(Q, C)$ such that $Q=C$ is finite and
  non-empty, or $Q$ is infinite and $C=\nat$, so $\iota$ is not
  surjective.  It also follows that $\iota \circ \varrho$ is not the
  identity map.
\end{example}

\section{\bf Open Questions}
\label{sec:open-questions}

In relation to Section~\ref{sec:stone-duals-locally}: (1) There are
non-spatial locally temperate frames, but is there a non-spatial
temperate frame?  (2) Is $\pt L$ locally strong sober for every weakly
temperate frame?  (3) Considering that every Hausdorff space is
locally strongly sober, is every $T_2$ frame locally temperate
(whatever the notion of being $T_2$ for frames you may wish to
consider)?






\section*{\bf Acknowledgments}
\label{sec:acknowledgments}

I thank Xiaodong Jia for suggesting Example~\ref{ex:Nab}, and for
finding a mistake in the original proof of Proposition~\ref{prop:coh}.

\bibliographystyle{plain}
\ifarxiv

\fi

\end{document}